\documentclass[11pt,reqno]{amsart}

\usepackage[utf8]{inputenc}
\input{mystyle.sty}

\graphicspath{{images/}}

\begin{document}

\title[Uniqueness of excited states]{Uniqueness of excited states to $-\Delta u + u -u^3=0$ in three dimensions}
\author{Alex Cohen, Zhenhao Li, Wilhelm Schlag}
\date{}

\address{Department of Mathematics \\ Massachusetts Institute of Technology \\ Cambridge, MA 02138, USA}
\email{alexcoh@mit.edu}
\address{Department of Mathematics \\ Massachusetts Institute of Technology \\ Cambridge, MA 02138, USA}
\email{zhenhao@mit.edu}
\address{Department of Mathematics \\ Yale University \\ New Haven, CT 06511, USA}
\email{wilhelm.schlag@yale.edu}

\thanks{
W.\ Schlag was partially supported by NSF grant DMS-1902691.
}

\begin{abstract}
   We prove the uniqueness of several excited states to the ODE $\ddot y(t) + \frac{2}{t} \dot y(t) + f(y(t)) = 0$, $y(0) = b$, and $\dot y(0) = 0$ for the model nonlinearity $f(y) = y^3 - y$. The $n$-th excited state is a solution with exactly $n$ zeros and which tends to $0$ as $t \to \infty$. These represent all smooth radial nonzero solutions to the PDE $\Delta u + f(u)= 0$ in $H^1$. We interpret the ODE as a damped oscillator governed by a double-well potential, and the result is proved via rigorous numerical analysis of the energy and variation of the solutions. More specifically, the problem of uniqueness can be formulated entirely in terms of inequalities on the solutions and their variation, and these inequalities can be verified numerically. 
\end{abstract}

\maketitle

\tableofcontents

\section{Introduction}

Consider the ODE
\begin{align}
    \ddot{y}(t) + \frac{2}{t}\dot{y}(t) + f(y(t)) = 0,\label{ODE:eq}\\
    y(0) = b,\ \dot{y}(0) = 0 \label{ODE:init_val}.
\end{align}
In this paper, we will only consider the model case $f(y) = y^3 - y$, but it will be convenient to use the more general notation for the nonlinearity. Smooth solutions to this ODE exist for all $t\ge0$ and any $b\in\R$, and they are unique. We denote them by $y_b$, or simply $y$. The singular coefficient at $t=0$ can be dealt with by a power series ansatz, or by Picard iteration. 
Solutions to this ODE correspond to radial smooth solutions of the PDE 
\EQ{\label{eq:PDE}
\Delta u + f(u) = 0
}
 in three dimensions, under the identification $t=r$, the radial variable.  Dynamics of~\eqref{ODE:eq} resembles particle motion in a double well as in Figure \ref{fig:double_well}, with time varying friction. 
\begin{figure}[ht]
\includegraphics[width=0.5\textwidth]{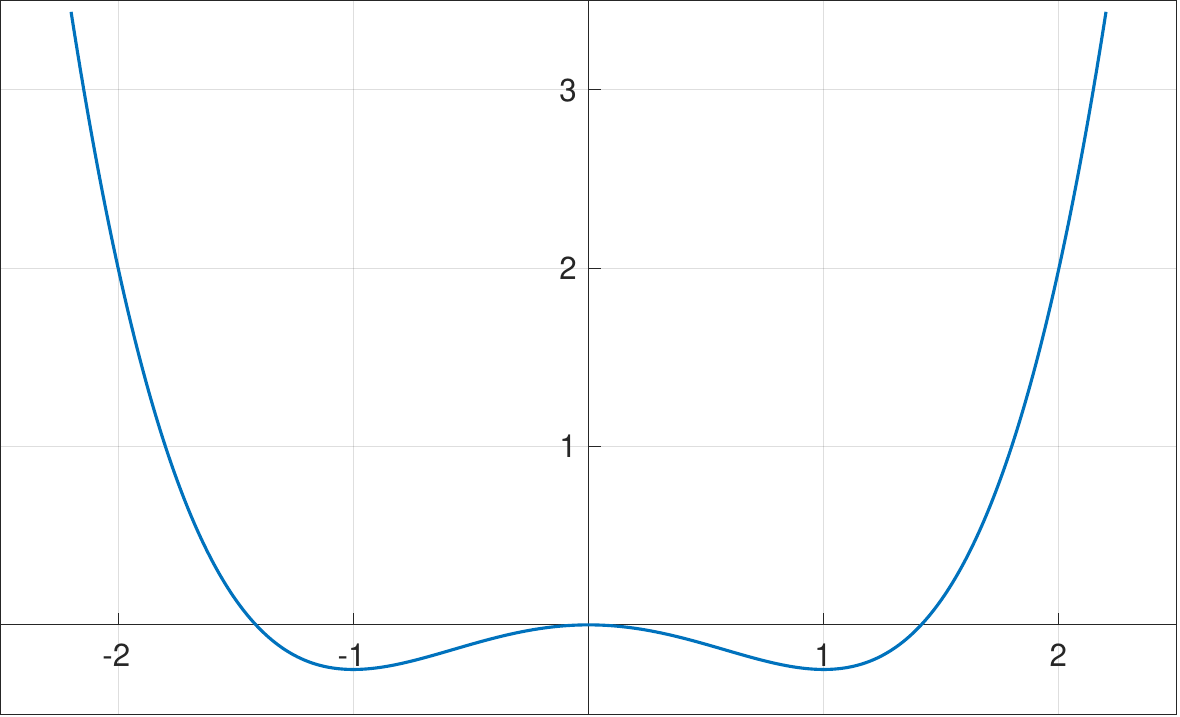}
\caption{Potential function $V(y) = \frac{1}{4}y^4 - \frac{1}{2}y^2$}
\label{fig:double_well}
\end{figure}
The qualitative behavior exhibits a trichotomy: we either have 
 \begin{itemize}
     \item $(y_b(t),\dot y_b(t)) \to (1,0)$
     \item $(y_b(t),\dot y_b(t)) \to (-1,0)$
     \item $(y_b(t),\dot y_b(t)) \to (0,0)$,
 \end{itemize}
see Section~\ref{sec:dyn} below. A \textit{bound state} is a nonzero solution with $y(t) \to 0$. Only these solutions give rise to nontrivial solutions $u\in H^1(\R^3)$ of the elliptic PDE~\eqref{eq:PDE}. A \textit{ground state} is a positive bound state, and an {\it $n^{th}$ excited state} is a bound state with precisely $n$ zero crossings. The $0^{th}$ bound state is the ground state.
Existence and uniqueness of these bound states have been investigated since the 1960s. In 1967, using a variational characterization, Ryder~\cite[Theorem II]{Ryder} showed the existence of both ground and excited states with any finite number of zero crossings. In 1972, Coffman~\cite[Section 6]{Coff} related Ryder's {\em characteristic values} to degree theory in infinite dimensions and Lyusternik-Schnirelman techniques.  Most importantly, \cite{Coff} also established uniqueness  of the ground state for the cubic case. For more general nonlinearities, ground state uniqueness was then shown by  Peletier, Serrin~\cite{PS1, PS2}, McLeod, Serrin~\cite{McLSerr}, Zhang~\cite{Z}, Kwong~\cite{K}, and finally in greatest generality by McLeod in 1992~\cite{M}. Clemons, Jones~\cite{CJ93} gave a different proof of McLeod's theorem based on the Emden-Folwer transformation and unstable manifold theory. In 1983, Berestycki and Lions~\cite{BerL1, BerL2} solved the existence problem of radial bound states for~\eqref{eq:PDE} for all $H^1$ subcritical nonlinearities $f(u)$ in all dimensions, see also the earlier work by Strauss~\cite{Strauss77}.  

However, uniqueness of excited states in the radial class, i.e., for the ODE~\eqref{ODE:eq},  remained open for most nonlinearities. In fact, in their 2012 text, Hastings, McLeod~\cite[Chapter 19]{HMcL} list this problem as one of three major open problems in nonlinear ODEs. We note that there has been some uniqueness results for specific nonlinearities; in 2005, Troy \cite{Troy} proved the uniqueness of the first excited state for a piecewise linear nonlinearity by analyzing the explicit solutions, and in 2009, Cort\'azar, Garc\'ia-Huidobro, and Yarur \cite{CGY} proved uniqueness of the first excited state with restrictions on $y f'(y)/f(y)$. However, neither cover the cubic nonlinearity. In this paper, we provide a rigorous computer-assisted proof of the uniqueness of the first excited state for the cubic nonlinearity. The proof technique combines analytical dynamics with the rigorous ODE solver \VNODELP{}, see Section~\ref{sec:vnode} and \cite{vnode}. The latter works with interval arithmetic and therefore does not compute precise solutions (which is impossible), but rather intervals  containing the solution at any given time. These inclusions accommodate all  errors incurred through floating point arithmetic, and are therefore themselves free of errors.  

\begin{theorem}\label{thm:first_excited_state_unique}
The first twenty excited states of ODE (\ref{ODE:eq}-\ref{ODE:init_val}) are unique for $f(y)=-y+y^3$.
\end{theorem}

The method is robust, and extends to both more general nonlinearities, as well as other dimensions. But we will leave the verification of this claim for another paper. The code involved in the proof is publicly available, see the GitHub repository \url{https://github.com/alexander-cohen/NLKG-Uniqueness-Prover}. The readers can verify uniqueness of higher excited states beyond the $20^{th}$ using the arguments of this paper. Computation time is the main obstacle to going further than the twentieth one, to which the authors chose to limit themselves. See Figure~\ref{fig:lim_pos_graph} for a graph of the limiting position of $y_b(T)$ as a function of $b$, up to the twentieth excited state. The rigorous numerical work done in this paper proves that this graph holds. 

The uniqueness property of the ground state soliton is of fundamental importance to the classification of its long-term evolution under the nonlinear cubic Schr\"odinger or Klein-Gordon flows. See for example~\cite{CazNLS} and~\cite{NakS}. The uniqueness property of the excited states should therefore also be seen as a bridge to dynamical results. As a first step, one needs to determine the spectrum of the linearized operator 
\[H = -\Delta +1 -3\phi^2 \]
in the radial subspace of $L^2(\R^3)$. Here $\phi$ is any radial bound state solution of the PDE~\eqref{eq:PDE}. If $\phi$ is the ground state, then it is known that the spectrum over the radial functions contains a unique negative eigenvalue, and no other discrete spectrum up to and including $0$ energy (nonradially, due to the translation symmetry, $0$ is an eigenvalue of multiplicity~$3$); see~\cite{NakS}. A particularly delicate question pertains to the shape of the spectrum in the interval $(0,1]$, including the threshold $1$ of the (absolutely) continuous spectrum. This was settled in~\cite{CHS} for the ground state of the cubic power in three dimensions. It turns out that $(0,1]$ is a spectral gap,  including  the threshold, which is not a resonance. Due to the absence of an explicit expression of the ground state soliton, the method of~\cite{CHS} depended on an approximation of this special solution. Note that without uniqueness such an approximation has no meaning. The authors intend to investigate the spectral problem of excited states in another publication using the methods of~\cite{CDSS}, which essentially require the uniqueness property of the special solution~$\phi$ (in the case of~\cite{CDSS} $\phi$ is the so-called Skyrmion).    
\begin{figure}
    \centering
    \includegraphics[width=0.8\textwidth]{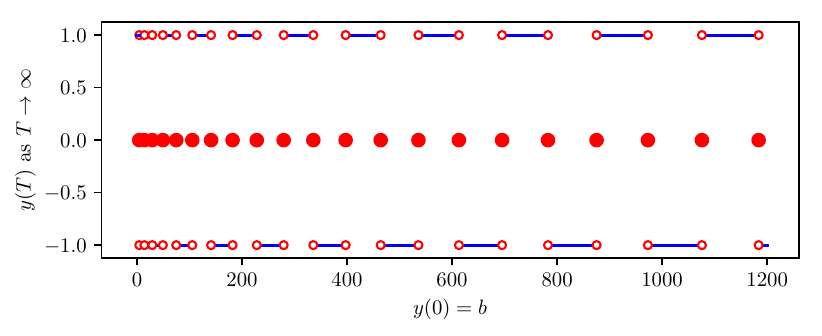}
    \caption{Limiting position $y_b(T)$ as $T\to \infty$, plotted as a function of the initial condition $b$ up to $b = 1200$. The solid red dots represent bound states, and this graph holds due to Theorem~\ref{thm:first_excited_state_unique} and the rigorous numerical work done in this paper.}
    \label{fig:lim_pos_graph}
\end{figure}

\section{Overview of approach}
\subsection{Toy example: finding zeros of a function}
Suppose we wish to find the number of zeros of the function $f$ as shown in the figure. Numerical computations make it clear that $f$ has exactly 3 zeros -- how can we use a computer to prove this rigorously?  

A first approach might be to find the approximate location of those zeros with reasonably  high precision, using floating point arithmetic. Say they lie at approximately $y_1, y_2, y_3$. Then we can use interval arithmetic to show rigorously that $f$ is bounded away from zero everywhere but the three small intervals $(y_1 - 1/100, y_1 + 1/100)$, $(y_2 - 1/100, y_2 + 1/100)$, $(y_3 - 1/100, y_3+1/100)$. Then, by interval arithmetic combined with the intermediate value theorem, $f$ has at least one zero in each of these intervals. 
Finally, to show that each of them contains \textit{at most} one zero, we can apply the mean value theorem. If we prove rigorously, using interval arithmetic, that $f'$ is bounded away from zero in those intervals, then uniqueness follows. 
\begin{figure}[ht]
\includegraphics[width=0.5\textwidth]{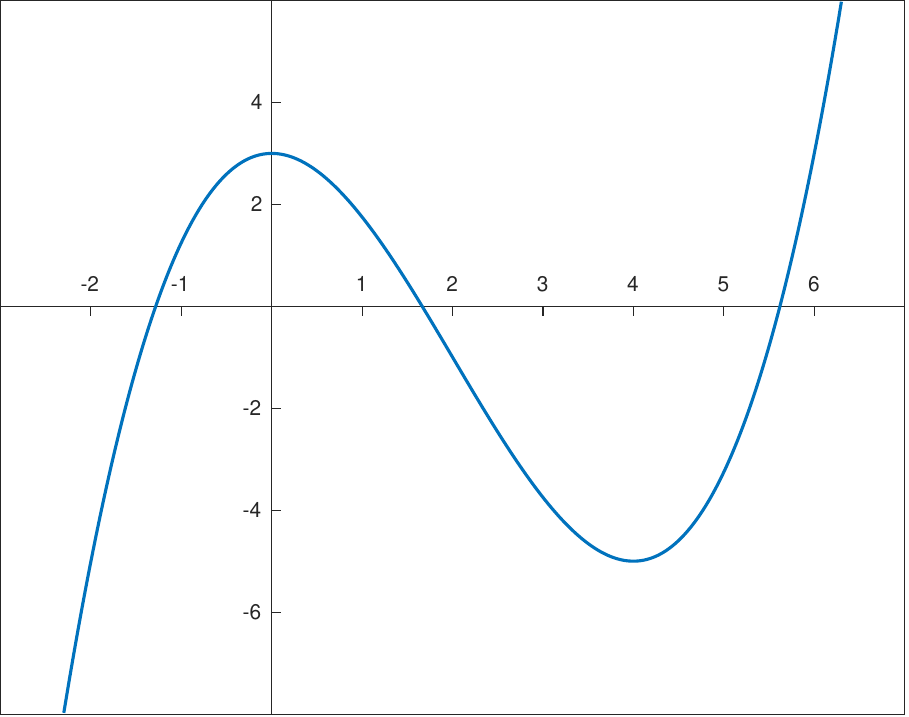}
\end{figure}
Notice that if $f$ has infinitely many zeros with a limit point, then $f'$ must be zero at that limit point. As expected, this method would break down in such a scenario. 

\subsection{Finding and isolating excited states}\label{sec:over}
We apply a similar idea  to the ODE (\ref{ODE:eq}). We now outline the approach by means of the ground state. Suppose we find numerically that the unique ground state should be at height $b_0 \approx 4.3373$. Using a rigorous ODE solver, we can prove that for all $b \in (1, b_0 - 0.001)$, $y_b(t) > 0$ up to some time $T$, and $E(T) < 0$. This will imply by analytical arguments, see the next section, that $y_b(t) \to 1$ and $y_b(t)$ is positive, so it is not a ground state. Similarly, we can show that for $b \in (b_0 + 0.001, 50)$, the solution passes over $y = 0$, and thus is not a ground state. It follows from a connectedness argument that there is some ground state in the interval $(b_0 - 0.001, b_0 + 0.001)$. To prove that there is \textit{exactly} one ground state in that interval, we find some large time $T$ such that $\delta_b(T), \dot{\delta}_b(T) < 0$ for all $b \in (b_0 - 0.001, b_0 + 0.001)$, where $\delta_b=\partial_b y_b$. This means that if $b_0^*$  is the actual ground state (rather than an approximation), for any $b > b_0^*$ in our interval, $y_b(T) < y_{b_0^*}(T)$ and $\dot{y}_b(T) < \dot{y}_{b_0^*}(T)$ by the mean value theorem. One can then prove assuming this condition that $y_b(T)$ crosses over zero and lands in the second well, see Lemma~\ref{lem:cross_over_get_trapped}. This will show that there is at most one ground state in the interval $(b_0 - 0.001, b_0 + 0.001)$. All that remains is showing that there is no ground state in the range $(50, \infty)$. To this end, we rescale the ODE (\ref{ODE:eq}) so that it takes the form $\ddot{w} + \frac{2}{t}\dot{w} + w^3 - b^{-2} w = 0$, $w(0)=1$, $\dot{w}(0)=0$. Again using \VNODELP{} we then show that the solution of this equation exhibits more than any given number of zeros provided $b$ is sufficiently large. This then implies the same for~$y_b$.

The same approach works just as well for excited states as it does for ground states.

\subsection{Approximating solutions via interval arithmetic}\label{sec:vnode}
We now outline our computational approach. Our main tool is the \VNODELP{} package for rigorous ODE solving. The supporting website\footnote{In the online version of this paper, click on \href{http://www.cas.mcmaster.ca/~nedialk/vnodelp/}{\VNODELP{}}} is at~\cite{mcmast1}, and the documentation\footnote{In the online version click on the hyperlink~\href{http://www.cas.mcmaster.ca/~nedialk/vnodelp/doc/vnode.pdf}{\it documentation}
	} is available at~\cite{mcmast2}. 

\VNODELP{} uses exact interval arithmetic, a toolset which allows for rigorous numerical computations. Rather than computing with floating point numbers as usual, interval arithmetic treats all values as \textit{intervals} of real numbers, of the form $\mb{a} = [a_1, a_2]$ where $a_1, a_2$ are machine representable floating point numbers. All mathematical operations are rounded properly so that any input within the original interval ends up within the output interval. The \VNODELP{} package combines interval arithmetic with ODE solving: given an initial value problem $\dot {y} = f(y,t)$ with initial values in an interval $\mb b$, a starting time interval $\mb t_1$, and an ending time interval $\mb t_2$, the package outputs an interval $\mb y$ such that for any $b \in \mb b$, $t_1 \in \mb t_1$, $t_2 \in \mb t_2$, $y_{b,t_1}(t_2) \in \mb y$. 

A difficulty in applying \VNODELP{} to our problem is that ODE (\ref{ODE:eq}) is singular at $t = 0$. To deal with this, we approximate $y_b(t)$ near $t = 0$ by  Picard iteration. We explicitly bound the error terms in this approximation so that we can rigorously obtain an interval containing $y_b(t_0), \dot{y}_b(t_0)$ for $t_0$ small. Then \VNODELP{} can be applied to this desingularized initial value problem, and all in all we we will have rigorous bounds on our solutions and quantities defined in terms of the solutions. 

Section~\ref{sec:numerics} explains in detail 
 how we use this software, and provides links to websites containing the code and all supporting data needed in the proof of our theorem. This will hopefully allow the reader to implement the methods of this paper in other related settings. 

\section{Analytical description of the damped oscillator dynamics}
\label{sec:dyn}

\subsection{Basic properties of the ODE}
It is an elementary property that smooth solutions of~\eqref{ODE:eq}, \eqref{ODE:init_val} exist for all times $t\ge0$; in fact, we will reestablish this fact below in passing. Taking it for granted, we note that the energy $$E(t) := \frac{1}{2} \dot{y}^2(t) + V(y(t))=\frac{1}{2} \dot{y}^2(t) + \frac14 y(t)^4 - \frac12 y(t)^2 $$
satisfies 
\[
\dot{E}(t) = -\frac{2}{t}\dot{y}^2(t)
\]
and thus $E(t)\le E(0)=V(b)$ for all times. In fact, $E(t)$ is strictly decreasing unless it is a constant and that can only happen for the unique stationary solutions $(y,\dot{y})$ equal $(0,0)$ or $(\pm1,0)$. In particular, if $V(b)\le 0$, then $E(t)<0$
for all $t>0$ unless $y(t)=0$ is a constant.  We will see below that this implies that $(y(t),\dot{y}(t))\to (1,0)$ as $t\to\infty$ (recall that we are assuming $b>0$). In other words, $(y,\dot{y})(t)$ approaches the minimum of the  potential well on the right of Figure~\ref{fig:double_well} and so $\inf_{t>0} y_b(t)>0$. The range of $b$ here is $0<b\le \sqrt{2}$. 

On the other hand, if $b>\sqrt{2}$, then $V(y(t))\le E(t)<E(0)=V(b)$ for all $t>0$ whence $$y(t)^2(y(t)^2-2)\le b^2(b^2-2)$$ and thus $|y(t)|\le b$ for all $t\ge0$. We will assume from now on that $b>\sqrt{2}$. 
Rewriting the initial value problem (\ref{ODE:eq}), \eqref{ODE:init_val} in the form
\begin{equation*}
    \frac{d}{dt}(t^2 \dot{y}(t)) + t^2f(y(t)) = 0,
\end{equation*}
where $f(y)=y^3-y$ throughout,  
we arrive at the integral equations
\EQ{\label{eq:ODE sys}
    y(t) &= b + \int_0^t \dot y(s)\, ds, \\
    \dot y(t) &= -\int_0^t \frac{s^2}{t^2} f(y(s))\, ds = \int_0^t \frac{s^2}{t^2} y(s)(1-y(s)^2)\, ds
}
For short times, we obtain a unique solution by the contraction mapping principle which is smooth near $t=0$.  
Picard iteration gives better constants, which is important for starting VNODE at some positive time. We shall determine the quantitative bounds in Section~\ref{sec:Picard}. 
But first we recall the equation of variation of (\ref{ODE:eq}) relative to the initial height~$b$. 

\subsection{The equation of variation}
 We let $\delta_b(t) := \frac{\partial}{\partial b} y_b(t)$. Then differentiating (\ref{ODE:eq}), $\delta_b(t)$ satisfies the ODE
\begin{equation*}
    \ddot{\delta} + \frac{2}{t}\dot{\delta} + f'(y)\delta = 0.
\end{equation*}
with initial conditions $\delta(0)=1$
 and $\dot{\delta}(0)=0$. Notice that the ODE for $\delta$ depends on the solution $y_b(t)$. Altogether, we can make one ODE in four variables that includes $y$ and $\delta$:
\begin{equation*}
    \frac{d}{dt} \begin{pmatrix}
         y \\ v_y \\ \delta \\ v_{\delta}
    \end{pmatrix} = \begin{pmatrix}
        v_y \\ -\frac{2}{t}v_y - f(y) \\ v_{\delta} \\ -\frac{2}{t}v_{\delta} - f'(y)\delta
    \end{pmatrix}.
\end{equation*}
with initial vector
\[
\begin{pmatrix}
        y \\ v_y \\ \delta \\ v_{\delta}
    \end{pmatrix}(0)=
    \begin{pmatrix}
        b \\ 0 \\ 1 \\ 0
    \end{pmatrix}
\]
Switching again to the variables $t^2 y(t)$, respectively, $t^2\delta(t)$,  and writing the resulting ODE  in integral form, this is equivalent to 
\EQ{
\label{eq:Pic4sys}
Z(t):= \begin{pmatrix}
        y \\ v_y \\ \delta \\ v_{\delta}
    \end{pmatrix}(t) &= \begin{pmatrix}
        b \\ 0 \\ 1 \\ 0
    \end{pmatrix} + \int_0^t \begin{pmatrix}
        v_y(s) \\  - t^{-2} s^2f(y(s)) \\ v_{\delta}(s) \\  - t^{-2}s^2 f'(y(s))\delta(s)
    \end{pmatrix}\, ds \\
    & = \begin{pmatrix}
        b \\ 0 \\ 1 \\ 0
    \end{pmatrix} + \int_0^t \begin{pmatrix}
        v_y(s) \\  t^{-2} s^2 y(s)(1-y(s)^2)  \\ v_{\delta}(s) \\   t^{-2}s^2 (1-3y(s)^2)\delta(s)
    \end{pmatrix}\, ds 
}
The first three Picard iterates of this system are 
\EQ{
\label{eq:Pic It}
Z_0(t) &= \begin{pmatrix}
        b \\ 0 \\ 1 \\ 0
    \end{pmatrix}, \quad  Z_1(t) = \begin{pmatrix}
        b \\ -\frac{t}{3} f(b) \\ 1 \\ -\frac{t}{3} f'(b)
        \end{pmatrix},\quad 
        Z_2(t) = \begin{pmatrix}
        b -\frac{t^2}{6} f(b) \\ -\frac{t}{3} f(b) \\ 1 -\frac{t^2}{6} f'(b)\\ -\frac{t}{3} f'(b)
    \end{pmatrix}
}

\subsection{Picard approximation}
\label{sec:Picard}

The purpose of this section is to compare the actual solution $Z(t)$ in~\eqref{eq:Pic4sys} to the second Picard iterate $Z_2(t)$ in~\eqref{eq:Pic It} which
we denote in the form 
\EQ{\label{eq:Z2 redef}
Z_2(t) =: \begin{pmatrix} 
       \tilde y(t)   \\  \dot {\tilde y}(t)  \\ \tilde \delta(t) \\  \dot {\tilde \delta}(t) 
        \end{pmatrix} = \begin{pmatrix} 
       \tilde y(t)   \\  \tilde v_y (t)  \\ \tilde \delta(t) \\ \tilde v_\delta  (t) 
        \end{pmatrix} 
}
In fact, we will prove the following inequalities on each of the four entries of this vector. 

\begin{lemma}
\label{lem:Z Z2}
Suppose $b\ge\sqrt{2}$. Then for all times $t\ge0$, 
\EQ{\label{eq:Z Z2}
   \tilde y(t) &\leq y(t) \leq  \tilde y(t) + \frac{f(b)f'(b)}{120}t^4\leq  \tilde y(t) + \frac{b^5}{40}t^4, \\
   \dot {\tilde y}(t) & \leq \dot y(t) \leq  \tilde y(t) + \frac{f(b)f'(b)}{30}t^3 \leq \dot {\tilde y}(t) + \frac{b^5}{10}t^3
}
For all $0\le t\le t_*$ with 
\EQ{\label{eq:t*def}
t_* := \min\Big(\sqrt{\frac{6(\sqrt{3}\,b -1)}{\sqrt{3}\, b(b^2-1)}}, \frac{\log 4}{\sqrt{3}\,b} \Big)
}
one has 
\EQ{\label{eq:Z Z2 2}
   \tilde \delta(t) &\leq \delta(t) \leq    \tilde \delta(t) + \frac{b^4}{8}t^4, \\
   \dot {\tilde \delta}(t) & \leq \dot \delta(t) \leq \dot {\tilde \delta}(t) + \frac{b^4}{2}t^3
}
\end{lemma}
\begin{proof}
Since energy is decreasing and $b \ge \sqrt{2}$, $|y(t)|\le b$ for all $t\ge0$. Note that 
\[f(b) \ge f(\sqrt{2}) > \frac{2}{3\sqrt{3}},\]
which is the absolute value of the local minima and maxima, so $|f(y)| \leq f(b)$ for all $|y| \leq b$. Therefore $|f(y(t))| \leq f(b)$ for all $t \geq 0$. Substituting this bound into \eqref{eq:ODE sys} yields
\begin{align}
    |\dot y(t)| &\leq \frac{t}{3} f(b), \label{eq:apriori_yderiv}\\
    0\leq b-y(t)  &\leq \frac{t^2}{6} f(b). \label{eq:apriori_y}
\end{align}
for all times $t\ge0$. 
Leveraging these bounds, we now compare the actual solution to its second Picard iterates as in~\eqref{eq:Pic It}.  In view of~\eqref{eq:Z2 redef},  $$\tilde y(t) = b - \frac{t^2}{6} f(b),\qquad \dot {\tilde y}(t) = - \frac{t}{3} f(b)$$ and we obtain via the mean value theorem that
\EQ{\nn 
   0\leq \dot y(t) - \dot {\tilde y}(t) &= \int_0^t \frac{s^2}{t^2}\big[ f(b) -f(y(s))\big]\, ds \leq \frac{f(b)f'(b)}{30}t^3, \\ 
   0\leq y(t) - \tilde y(t) &\leq \frac{f(b)f'(b)}{120}t^4,
}
where we used that $|f'(y)|\le f'(b)=3b^2-1$ for all $|y|\leq b$. 

The last two rows of \eqref{eq:Pic4sys} imply that
\EQ{\nn
|\delta(t)-1| &\leq \int_0^t |\dot{\delta}(s)|\, ds \\
|\dot{\delta}(t)| & \leq t^{-2} f'(b)\int_0^t s^2 |\delta(s)| \, ds \\ &\leq \frac{t}{3}f'(b) + t^{-2} f'(b)\int_0^t s^2 |\delta(s)-1|\, ds\\
&\leq tb^2 + 3b^2 \int_0^t |\delta(s)-1|\, ds
}
whence $h(t):=|\dot{\delta}(t)| +\mu |\delta(t)-1|$ with $\mu:=\sqrt{3}\, b$ satisfies
\EQ{\label{eq:hbd} 
 h(t) &\leq tb^2 +  \mu\int_0^t h(s) \, ds \\
 h(t) &\leq  \frac{b^2}{\mu}(e^{\mu t}-1)
}
We infer from the last two rows of \eqref{eq:Pic4sys} and~\eqref{eq:Z2 redef} that
\EQ{\label{eq:deltdel}
\delta(t)-\tilde\delta(t) & = \int_0^t (v_\delta(s) - \tilde v_\delta(s))\, ds \\
v_\delta(t) - \tilde v_\delta(t) &= t^{-2}\int_0^t s^2 (f'(b) - f'(y(s))\delta(s))\, ds \\
& = t^{-2}\int_0^t s^2 \big[f'(b) - f'(y(s)) + f'(y(s)) (1-\delta(s))\big]\, ds 
}
as well as 
\EQ{\label{eq:pos}
1-\delta(t) &= -\int_0^t v_\delta(s)\, ds \\
-v_\delta(t) & = t^{-2}\int_0^t s^2  f'(y(s))\delta(s) \, ds
}
Let $t_*>0$ be such that $f'(y(t))\ge0$ and $\delta(t)\ge0$ for all $0\le t\le t_*$. Then by~\eqref{eq:pos}, $\delta(t)\le1$ for those times and thus by~\eqref{eq:deltdel} 
\EQ{\nn 
\delta(t)-\tilde\delta(t) &\ge 0 \\
v_\delta(t) - \tilde v_\delta(t) &\ge 0 
}
for all $0\le t\le t_*$. By \eqref{eq:hbd}, we have $\delta(t)\ge0$ as long as
\[
e^{\mu t}\le 4, \quad t\le \frac{\log 4}{\sqrt{3}\,b}
\]
Moreover,  $f'(y(t))\ge0$ as long as $\sqrt{3}\, y(t)\ge 1$ which by \eqref{eq:Z Z2} holds provided  
\[
\sqrt{3}\, \tilde y(t)\ge 1,\quad t\le \sqrt{\frac{6(\sqrt{3}\,b -1)}{\sqrt{3}\, b(b^2-1)}}
\]
whence in summary gives \eqref{eq:t*def} and the lower bounds in~\eqref{eq:Z Z2 2}.  
For the upper bound, note that \eqref{eq:pos}, respectively,~\eqref{eq:Z Z2} imply that \EQ{\nn 
-v_\delta(t)&\leq t b^2, \quad 1-\delta(t) \leq \frac12 t^2 b^2 \\
b-y(t) &\leq \frac{t^2}{6} b^3 
}
Inserting these bounds into \eqref{eq:deltdel} yields by the mean value theorem
\EQ{\nn
v_\delta(t) - \tilde v_\delta(t) 
& \leq t^{-2}\int_0^t s^2 (b^4 s^2 + 3b^4 s^2/2)\, ds \\
&\leq \frac12 b^4 t^3 \\
\delta(t)-\tilde\delta(t) & = \int_0^t (v_\delta(s) - \tilde v_\delta(s))\, ds \\
&\leq \frac18 b^4 t^4
}
as claimed. 
\end{proof}

\subsection{The equation at infinity}\label{sec:eq_at_infty}

As explained in Section~\ref{sec:over}, to prove uniqueness of the first excited state  we will need to show that all $y_b$  have at least two crossings for all sufficiently large~$b$. For the second  excited state, we need to do the same with three crossings, and so on.  This will be  accomplished by means of the following lemma. 

\begin{lemma}
\label{lem:ode at infty}
Let $y(t)$ be a solution to ODE (\ref{ODE:eq}-\ref{ODE:init_val}). Let $w(s) = \frac{1}{b}y(s/b)$. Then $w$ satisfies
\begin{align}
    \ddot{w} + \frac{2}{s} \dot{w} + w^3 - \beta^2 w = 0; \label{ODEv:eq}\\
    w(0) = 1,\quad \dot{w}(0) = 0. \label{ODEv:init_val}
\end{align}
where $\beta := \frac{1}{b}$. 
\end{lemma}
\begin{proof}
Immediate by scaling. 
\end{proof}

We will analyse this initial value problem with \VNODELP{}, but as before can only start at positive times rather than at $t=0$. 
The analogue of Lemma~\ref{lem:Z Z2} is the following. We only need to approximate the ODE in~\eqref{ODEv:eq}. Indeed, since the initial condition is fixed, the equation of variations does not arise. 

\begin{lemma}
\label{lem:w approx}
Suppose $0<\beta\le\frac{1}{10}$, and let $\tilde w(t) := 1 - \frac{1-\beta^2}{6} t^2$, and $\dot{\tilde w}(t) := -\frac{1-\beta^2}{3}t$. Then for all times $t\ge0$, 
\EQ{\label{eq:wtilw}
   \tilde w(t) &\leq w(t) \leq  \tilde w(t) + \frac{t^4}{40}, \\
   \dot {\tilde w}(t) & \leq \dot w(t) \leq  \dot {\tilde w}(t) + \frac{t^3}{10}
}
\end{lemma}
\begin{proof}
We write the equation~\eqref{ODEv:eq} in the form
\[
\frac{d}{dt}(t^2 \dot w(t)) = -t^2 f_\beta(w(t))
\]
with 
$$f_\beta(w) := w^3 - \beta^2 w= w(w^2-\beta^2)$$ 
Solutions are global, and the energy takes the form 
\[
E_\beta(t) = \frac12 \dot w^2(t) + V_\beta(w(t)),\quad V_\beta(w) = \frac14 w^4 - \frac12\beta^2 w^2
\]
which is nonincreasing as before. Thus, $V_\beta(w(t))\le E_\beta(t)\leq V_\beta(1)=\frac14-\frac{\beta^2}{2}$ whence $|w(t)|\le 1$ for all times. 
The integral formulation of the initial value problem for $w$ is of the form \EQ{\label{eq:ODE w sys}
    w(t) &= 1 + \int_0^t \dot w(s)\, ds, \\
    \dot w(t) &= -\int_0^t \frac{s^2}{t^2} f_\beta(w(s))\, ds = \int_0^t \frac{s^2}{t^2} w(s)(\beta^2-w(s)^2)\, ds
}
Inserting $w=1$ into the right-hand side of the second row of~\eqref{eq:ODE w sys} gives $\dot{\tilde w}(t) := -\frac{f_\beta(1)}{3}t$, and $\tilde w$
is obtained by inserting this expression into the right-hand side of the first row of~\eqref{eq:ODE w sys}. These are precisely the approximate solutions appearing in the formulation of the lemma. The stated bounds are now obtained as in Lemma~\ref{lem:Z Z2} and we leave the details to the reader. 
\end{proof}

\subsection{Limit sets and convergence theorems}
As we have already noted, 
an important quantity associated with the equation (\ref{ODE:eq}) is the energy $E(y, \dot{y}) = \frac{1}{2} \dot{y}^2 + V(y)$, where $V(y) = \int_0^y f(y)$ is the potential energy. Explicitly, $V(y) = \frac{y^4}{4} - \frac{y^2}{2}$ resembles a double well as in the first figure above. 
Were we to modify our ODE to $\ddot{y} + f(y) = 0$, then the energy would be preserved. The term $\frac{2}{t} \dot{y}$ adds a time dependent frictional force, so energy decreases monotonically:
\begin{equation*}
    \dot{E}(t) = \dot{y}\ddot{y}(t) + f(y(t))\dot{y}(t) = -\frac{2\dot{y}^2(t)}{t}.
\end{equation*}
The interpretation of the radial form of the PDE~\eqref{eq:PDE} as a damped oscillator with the role of time being played by the radial variable is of essential importance in this section. Tao~\cite{TCBMS} emphasized this already in his exposition of ground state uniqueness, but here we will rely  on this interpretation even more heavily. In particular, the proof of the long-term trichotomy given by the solution vector of the main ODE approaching one of the three critical  points of the potential follows the dynamical argument in the damped oscillator paper~\cite{CEG}. 

The following lemma determines the $\omega$-limit set of every trajectory in phase space. The lemma combined with the monotonicty of the energy will help us determine the desired long-term trichotomy.

\begin{lemma}\label{lemma:convergent_timestamps}
If $y(t)=y_b(t)$ is the global solution to the initial value problem~\eqref{ODE:eq}, \eqref{ODE:init_val}, then there exists an increasing unbounded  sequence $\{t_j\}$ such that $(y(t_j), \dot y(t_j)) \to (0, 0)$ or $(\pm 1, 0)$ as $j \to \infty$. 
\end{lemma}
\begin{proof}
From boundedness of the energy, we see that 
\[\sum_{n = 1}^\infty \frac{1}{n} I_n \le \int_0^\infty \frac{\dot y(t)^2}{t} \, dt < \infty,\quad I_n := \int_{n - 1}^n \dot y(t)^2\, dt.\]
 Therefore,  $I_{n_j} \to 0$ as $j \to \infty$ for some subsequence. We can  pick $t_j \in (n_j - 1, n_j)$ so that $\dot y(t_j) \to 0$. Since $E(t)$ and $y(t)$ are bounded,  $\ddot y$ is bounded, and differentiating~\eqref{ODE:eq}, we then see that $\dddot y$ is also bounded. Therefore $I_{n_j} \to 0$ implies that $\ddot y(t_j) \to 0$. This implies that 
\[|f(y(t_j))| \le |\ddot y(t_j)| + \frac{2}{t} |\dot y(t_j)| \to 0,\]
so there must be a subsequence of $y(t_{n_j})$ that converges either to $0$ or $1$.
\end{proof}

Next, we establish that each trajectory must converge to the point in its limit set, cf.\ the convergence theorems in~\cite{CEG}.  

\begin{lemma}\label{lemma:limit_behav}
Either $y_b(t) \to -1$, or $y_b(t) \to 0$, or $y_b(t) \to 1$ as $t\to\infty$. 
In all cases $\dot{y}(t)\to0$. 
\end{lemma}
\begin{proof}
Since $E(t)$ is monotonically decreasing, the limit $\lim_{t\to\infty}E(t)$ exists as a real number, which is either  negative or non-negative. In the former case, there must be a sequence $t_j$ so that $(y(t_j), \dot y(t_j)) \to (\pm 1, 0)$ by Lemma~\ref{lemma:convergent_timestamps}. Monotonicity of the energy then implies that $E(t) = E(y_b(t), \dot y_b(t))$ tends toward the global minimum value of the potential energy, which means that $(y_b(t), \dot y_b(t)) \to (\pm 1, 0)$. 

If the limit is non-negative, then  Lemma~\ref{lemma:convergent_timestamps} implies that $E(t) \to 0$ as $t \to \infty$. Suppose $y(t)$ does not converge to $0$. Let $\tau_j$ denote the $j$-th time at which $\dot y(\tau_j) = 0$, and if $y(t)$ does not tend to $0$, then $\{\tau_j\}$ is an infinite sequence. We will show that this leads to a contradiction since too much energy will be lost in each oscillation. To do so, we first upper bound $\tau_{j + 1} - \tau_{j}$. Assume without loss of generality that $\dot y > 0$ between $\tau_{j}$ and $\tau_{j + 1}$. We have $V(y) \le - \frac{1}{4} y^2$ for $y \in (-1, 1)$. Let $\tau_j < t_1 < t_2 < \tau_{j + 1}$ so that $y(t_1) = -1$ and $y(t_2) = 1$. In particular, the portion of the trajectory between $t_1$ and $t_2$ is the part of the trajectory going over the hill in the potential, which should be the most time-consuming part of the trajectory, and indeed,
\begin{align}
    \int_{t_1}^{t_2} 1\, dt &= \int_{-1}^1 \frac{1}{y'}\, dy \nonumber     = \int_{-1}^1 \frac{1}{\sqrt{2(E(y) - V(y))}}\, dy \nonumber \\
    &\le 2 \int_{0}^1 \frac{1}{\sqrt{2E(t_2) + y^2/2}}\, dy \nonumber \\
    &\le 2\int_0^{\sqrt{2E(t_2)}} \frac{1}{\sqrt{2E(t_2)}}\, dy + 2\sqrt{2} \int_{\sqrt{2E(t_2)}}^1 \frac{1}{ y}\, dy \nonumber \\
    &\lesssim - \log E(t_2), \label{eq:time_near_zero}
\end{align}
assuming that $\tau_j$ is sufficiently large so that $0<E(t_2)\ll 1$. Note that from the  energy, $\dot y(t)$ can only reverse sign if $|y(t)|> \sqrt{2}$. Since the energy is always positive,
\[E(t_2) \ge \int_{t_2}^{\tau_{j + 1}} \frac{2 \dot y(t)^2}{t} \, dt \ge \frac{1}{\tau_{j + 1}} \int_1^{\sqrt{2}} 2\sqrt{2(E(y) - V(y))}\, dy \gtrsim \frac{1}{\tau_{j + 1}}.\]
Substituting this into \eqref{eq:time_near_zero}, we find that
\begin{equation}\label{eq:to_be_improved}
t_2 - t_1 \lesssim \log \tau_{j + 1}
\end{equation}
Finally, we show that for small energy, the time spent by $y(t)$ in one oscillation outside the interval $(-1, 1)$ is uniformly bounded by some constant. Fix some $0 < \varepsilon \ll \sqrt{2} - 1$ so that $|f(y)| \ge \alpha > 0$ for all $y \in B_\varepsilon(\pm \sqrt{2})$, the $\varepsilon$-neighborhoods of $\pm\sqrt{2}$. Then there exists a sufficiently large $T$ so that $\sqrt{2} < |y(\tau_j)| < \sqrt{2} + \varepsilon$ for all $\tau_j > T$ and that $|2\dot y(t)| / t < \alpha/2$ for all $t > T$. This means that if $\tau_j > T$ and $\dot y(t) > 0$ for $t \in (\tau_j, \tau_{j + 1})$, then 
\begin{align*}
    &\ddot y(t) = -f(y(t)) - \frac{2}{t} \dot y(t) \ge \frac{\alpha}{2} &\text{when } y(t) \in B_\varepsilon(- \sqrt{2}), \\
    &\ddot y(t) = -f(y(t)) - \frac{2}{t} \dot y(t) \le -\alpha &\text{when } y(t) \in B_\varepsilon(+\sqrt{2}).
\end{align*}
In other words, between $\tau_j$ and $\tau_{j + 1}$, the initial acceleration and final deceleration are both uniformly lower bounded. Then there is a uniform constant bounding the time spent by the part of the trajectory in $B_\varepsilon(\pm \sqrt{2})$. Outside both $B_\varepsilon(\pm \sqrt{2})$ and the interval $(-1, 1)$, the velocity is uniformly lower bounded, so there is a uniform constant bounding the time in that region as well.

Therefore, (\ref{eq:to_be_improved}) can in fact be improved to $\tau_{j + 1} - \tau_{j} \lesssim \log \tau_{j + 1}$, so $\tau_{j} \lesssim j \log \tau_j$. One first reads off $\tau_j \lesssim j^2$, and then applies this inequality once more to conclude 
\[\tau_j \lesssim j \log j.\]
The cummulative loss in energy starting from some sufficiently large time  $\tau_N$ is  therefore
\[\int_{\tau_N}^\infty \frac{\dot y^2(t)}{t}\, dt \ge \sum_{j = N}^\infty \frac{1}{\tau_{j + 1}} \int_{\tau_j}^{\tau_{j + 1}} \dot y^2(t)\, dt \gtrsim \sum_{j = N}^\infty \frac{1}{j \log j},\]
which is not finite, a contradiction.
\end{proof}

\subsection{Passing over the saddle}
\label{sec:passing_over_hill}

We now turn to a lemma which establishes the following natural property: consider the value $0<y(T)= \varepsilon \ll 1$  of a bound state solution with $T$ so large that $y(t)>0$ for all $t>T$. Then any other $y_b$ with $y_b(T)\in (0,\varepsilon)$ and $\dot y_b(T)<\dot y(T)$ needs to cross $0$ after time $T$. For simplicity, we prove the lemma for $f(y) = y^3 - y$, but it is easy to see that it works for many nonlinearities via the same argument.

\begin{lemma}\label{lem:above_does_cross}
Suppose $b^* \in (0, \infty)$ is a bound state, and without loss of generality assume $y_{b^*}(t)$ approaches $0$ from the right. That is, $\dot{y}_{b^*}(t) < 0$ for all $t \geq T$, for some $T$. Then $y_{b^*}$ has no more zero crossings after time $T$, and increasing $T$ if necessary, we may assume $0 < y_{b^*}(T) \leq 1/\sqrt{3}$. If $y_b(t)$ is another solution with $0<y_b(T) < y_{b^*}(T)$ and $\dot{y}_b(T) < \dot{y}_{b^*}(T)$, then $y_b(t)$ has a zero crossing after time $T$.
\end{lemma}

\begin{proof}
Let $s(t) = y_{b^*}(t) - y_b(t)$. Then 
\begin{equation}\label{eq:difference}
    \ddot{s}(t) + \frac{2}{t} \dot {s}(t) = f(y_b(t)) - f(y_{b^*}(t)).
\end{equation}
At $t = T$, $s(T) > 0$ and $\dot s(T) > 0$. If $y_b(t)$ does not cross zero for any $t > T$, $y_b(t) \to 0$ or $1$. This means that $s(t) \to 0$ or $-1$. In either case, $s(t)$ must reach a maximum after $t = T$, so there exists a $t_* > T$ so that $\dot s(t_*) = 0$, $s(t_*) > 0$ and $\ddot s(t_*) \le 0$. Then by (\ref{eq:difference}), 
\begin{equation}\label{eq:diffcontradiction}
    f(y_b(t_*)) - f(y_{b^*}(t_*)) \le 0
\end{equation}
It is clear that $f(y)$ is strictly decreasing for $y \in (0, 1/\sqrt{3})$, so when $s(t_*) > 0$, (\ref{eq:diffcontradiction}) leads to a contradiction with the assumption $y_{b^*}(T) \le 1/\sqrt{3}$
\end{proof}
Note that the only property of $f$ we used is that $f'(0) < 0$, which holds for all nonlinearities associated with a double-well potential.


The next lemma provides a sufficient condition under which the trajectory will pass over the hill and be trapped in the following well. The underlying mechanism is the consumption of energy due to a necessary oscillation around the left well. If this amount exceeds the energy present at the pass over the saddle at $y=0$, then the remaining energy is negative, ensuring trapping. The lemma will ensure that if $y_{b^*}(t)$ is a bound state, then for initial values $b \in (b^*, b^* + \varepsilon)$ for some small $\epsilon$, $y_b(t)$ will necessarily fall into the following potential well.  

\begin{lemma}\label{lem:cross_over_get_trapped}
Suppose $y(t)$ is a solution of (\ref{ODE:eq}), \eqref{ODE:init_val} such that for some $T > 0$, 
\begin{align*}
    0 \leq y(T) &< \frac{1}{2}, \\
    \dot{y}(T) &< 0, \\
    0 < E(T) &< \frac{1}{4}, \\
    E(T)\left(T - 2\ln E(T) + \frac{3}{2}\right) &< \frac{3}{8}.
\end{align*}
Then if $y(t)$ has a zero after (or at) time $T$, it must proceed to fall into the left well. That is, $y(t) \to -1$, and $y(t)$ has no further zero crossings. 
\end{lemma}
\begin{proof}
Suppose $y(t)$ has another zero crossing, say the minimal time $t\geq T$ with this property is $t_0\geq T$. Then $\dot{y}(t_0)<0$ and there can be no reversal in the sign of $\dot{y}(t)$ until after $y(t)$ has passed~$-1$. So we  can define $T_1 > T$ to be the first time after $T$ at which $y(T_1) = -1/2$. 

Suppose $y(t)$ does not fall into the left well; in this situation, $E(T_1) = \alpha > 0$. Then we must have $E(T) > \alpha$. Let $t(s)$, $-1/2 < s < 1/2$, be the nearest time after/before $T$ such that $y(t(s)) = s$. Then  we have (recall $0<\alpha<\frac{1}{4}$)
\begin{align*}
    T_1 - T &< \int_{-1/2}^{1/2} \frac{1}{|\dot{y}(t(s))|}\, ds
    = \int_{-1/2}^{1/2} \frac{1}{\sqrt{2(E(t(s)) - V(s))}}\, ds \\
    &< \int_{-1/2}^{1/2} \frac{1}{\sqrt{2\alpha + s^2 - s^4/2}}\, ds     < \int_{-1/2}^{1/2} \frac{1}{\sqrt{2\alpha + s^2/2}}\, ds \\
    &< 2\int_{0}^{\sqrt{\alpha}} \frac{1}{\sqrt{2\alpha}} \, ds + 2\sqrt{2}\int_{\sqrt{\alpha}}^{1/2} \frac{1}{s}\, ds \\
    &= \sqrt{2} - 4\ln(2) - 2\ln(\alpha) < -2\ln(\alpha).
\end{align*}
Thus $T_1 < T - 2\ln(\alpha)$. Next, we observe that if the conditions of the lemma are satisfied, then more than $\alpha$ energy is lost in going from $y = -1/2$ to $y = -1$. Letting $t(s)$ be as before and assuming $y(t)$ does not fall into the left well, $E(t(s)) > 0$ for $-1 < s < -1/2$. Using $\frac{dE}{ds} = \frac{1}{\dot{y}}\frac{dE}{dt} = \frac{2|\dot{y}|}{t}$, we have
\begin{align*}
    \Delta E &= 2\int_{-1}^{-1/2} \frac{|\dot{y}(t(s))|}{t(s)}\, ds \\
    &= 2\int_{-1}^{-1/2} \frac{\sqrt{2(E(t(s)) - V(s))}}{t(s)}\, ds \\
    &> 2\int_{-1}^{-1/2}\frac{\sqrt{s^2 - s^4/2}}{t(s)}\, ds 
    > \frac{1}{T_2}\int_{-1}^{-1/2} |s|\, ds 
    = \frac{3}{8T_2}
\end{align*}
where $T_2 > T$ is the first time at which $y = -1$. Now we have
\begin{align*}
    T_2 - T_1 &= \int_{-1}^{-1/2} \frac{1}{|\dot{y}(t(s))|}\, ds \\
    &= \int_{-1}^{-1/2}\frac{1}{\sqrt{2(E(t(s)) - V(s))}}\, ds \\
    &< \int_{-1}^{-1/2}\frac{1}{\sqrt{y^2 - y^4/2}}\, dy \\
    &< 2\int_{-1}^{-1/2}\frac{1}{|y|}\, dy 
    < \frac{3}{2}.
\end{align*}
Altogether, we have
\begin{equation*}
    \Delta E > \frac{3}{8}\frac{1}{T - 2\ln(\alpha) + \frac{3}{2}},
\end{equation*}
and if $\Delta E > \alpha$, then $E(T_2) < 0$, and the particle falls into the left well. This occurs when
\begin{equation*}
    \alpha\left(T - 2\ln(\alpha) + \frac{3}{2}\right) < \frac{3}{8}
\end{equation*}
Because $\alpha < E(T)$ and $\alpha \ln(\alpha)$ is monotone decreasing for $0<\alpha < 1/4$, in the situation of the lemma, if 
\begin{equation*}
    TE(T) - 2E(T)\ln E(T) + \frac{3}{2}E(T) < \frac{3}{8}
\end{equation*}
then the particle must fall into the left well if it crosses zero after time $T$. 
\end{proof}

\section{Proof of Theorem~\ref{thm:first_excited_state_unique}}
\subsection{Outline of proof}

Theorem~\ref{thm:first_excited_state_unique} is proved by running a \texttt{C++} computer program which combines the rigorous numerics of \VNODELP{} with the analytical lemmas of the preceding section. This code is divided into two parts: a \textit{planning} section, and a \textit{proving} section. The planning section of the code creates a plan for proving the first several bound states are unique, and the proving section executes this plan and outputs a rigorous proof of uniqueness. Separating these two sections is advantageous because only the proving section must be mathematically rigorous, so only that part of the code needs to be checked for correctness. The planning section can be modified without fear of compromising the rigour of the code. 

In what follows we treat \VNODELP{} as a black box that takes in an input interval $\mb{b} = (b_1, b_2)$ and a time interval $\mb{t} = (t_1, t_2)$, and outputs an interval $\mb y_{\mb b}(\mb t) = (y_1, y_2) \times (\dot y_1, \dot y_2) \times (\delta_1, \delta_2) \times (\dot \delta_1, \dot \delta_2)$ which contains $y_b(t)$ for any $b \in (b_1, b_2)$ and $t \in (t_1, t_2)$. We can also integrate the equation at infinity (\ref{ODEv:eq}) rigorously. To implement this functionality, we use the explicit error bounds given in Lemma \ref{lem:Z Z2} to move past the singularity at $t = 0$. For instance, we may pick $\mb t_0 = (0.1, 0.101)$ and then use those error bounds to find $\mb y_0$ a vector of four intervals which contain any $y_b(t)$, $b \in (b_1, b_2)$, and $t \in (0.1, 0.101)$. At this point we may input the starting intervals $\mb y_0, \mb t_0$ directly into \VNODELP{}, which rigorously integrates to the desired ending time. See Figure~\ref{fig:pos_vnode_intervals} for a depiction of \VNODELP{} integration with solution intervals.

\begin{figure}
    \centering
    \includegraphics[width=0.8\textwidth,trim={0 0 0 0.7cm}, clip]{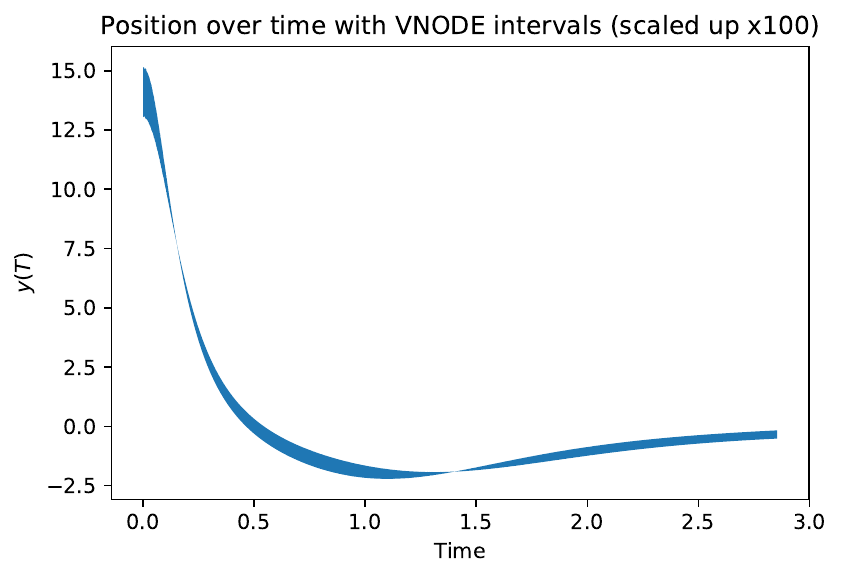}
    \caption{\VNODELP{} numerical integration with solution intervals scaled up $\times 100$. The ``pinch points'' of near zero $y$-uncertainty occur when $\delta = y_b(t) \sim 0$, and although it is not shown here, the $\dot y$-uncertainty is larger at these points.}
    \label{fig:pos_vnode_intervals}
\end{figure}

We now describe our procedure for the ground state and first excited state, before describing the planning and proving sections in detail. Bound states can only occur in the range $b \in (\sqrt{2}, \infty)$. To prove the ground state is unique, we split this range into four intersecting intervals, $I_1, I_2, I_3, I_4$. For instance, we can take:
\begin{equation*}
    I_1 = (1.4, 4.26),\ I_2 = (4.25, 4.43),\ I_3 = (4.42, 6.32),\ I_4 = (6.31, \infty).
\end{equation*}
Now, numerical evidence shows that the ground state occurs in the range $I_2$. So, in the range $I_1$, the solution will eventually fall into the right energy well. We use \VNODELP{} to prove this by splitting $I_1$ into smaller chunks, and verifying that in each of these chunks the energy of the solution eventually falls below zero. We deal with the range $I_3$ in the same way, by showing that for all $b\in I_3$, $y_b(t)$ eventually has negative energy. In the interval $I_4$, the solution should always have at least one zero crossing. We prove this using the equation at infinity (\ref{eq:ODE w sys}) as discussed in \S\ref{sec:eq_at_infty}. The infinite range $b \in I_4$ corresponds to the finite range $\beta \in (0, 0.16)$, so by splitting this range into small chunks and verifying that in these chunks $w(t)$ is eventually negative, we prove that $y_b(t)$ eventually crosses zero for all $b \in I_4$. Notice that we could have replaced the interval $I_4$ by $I_3 \cup I_4$, and it would still be true that in this range there is at least one zero crossing. We handle with $I_3$ separately because the numerics of ODE (\ref{eq:ODE w sys}) are delicate near bound states, so $I_3$ acts as a \textit{buffer interval}. 

The only range left is $I_2$, which actually contains the ground state. We must show that $I_2$ contains at most one ground state, and that it contains no first or higher excited states. To this end, we use Lemmas \ref{lem:above_does_cross} and \ref{lem:cross_over_get_trapped} respectively. At some time $T > 0$, say $T = 6$, any $y_b(T)$ for $b \in I_2$ will be positive, moving in the negative direction, and small in magnitude. We use \VNODELP{} to prove that $\delta_b(T), \dot{\delta}_b(T) < 0$ for $b \in I_2$. Let $b_0 \in I_2$ be a ground state. Then for any $b > b_0$, the mean value theorem implies that $y_b(T)< y_{b_0}(T)$ and $\dot y_b(T) < \dot y_{b_0}(T)$. By Lemma \ref{lem:above_does_cross}, $y_b(t)$ must cross zero. It follows that there is at most one ground state in $I_2$. Next, we check that the conditions of Lemma \ref{lem:cross_over_get_trapped} are satisfied for all $y_b(T)$, $b \in I_2$. This implies that if any solution $y_b(t)$ does cross zero, it must fall into the left energy well, and cannot be a higher excited state. Altogether this shows that there is at most one ground state for the ODE (\ref{eq:ODE sys}), as desired. It also follows from this analysis that the ground state exists, so we have successfully shown the ground state exists and is unique. 

We note a subtlety in this argument. A pathological issue would be that by time $T = 6$, $y_b(T)$ crossed all the way to the left well, come back to the right well, and then started to approach $y = 0$ from the right. This would kill our later attempt to prove that the second excited state is unique, because we would miss a second excited state in $I_2$. To deal with this, we use energy considerations to bound $|\dot y_b(t)|$, and we make small enough time steps with \VNODELP{} so that the solution cannot cross zero twice in between time steps. Then, we can be sure that the number of zero crossings observed by \VNODELP{} up to some time $T$ is the actual number of zero crossings for all solutions in our initial interval $\mathbf{b}$, up to time $T$. 

To prove the ground state is unique, we split the range $(\sqrt{2}, \infty)$ into subintervals to which we applied three different proof methods. The method for $I_1$ and $I_3$ was \FALL{}: we proved that the solution eventually has negative energy and thus cannot be a bound state. The method for $I_4$ was \INFTYCROSSESMANY{}, we used the equation at infinity to show that there are sufficiently many zero crossings and thus no ground states. The method for $I_2$ was \BOUNDSTATEGOOD{}, we used the analytical Lemmas \ref{lem:above_does_cross}, \ref{lem:cross_over_get_trapped} to show that there was at most one ground state and no other bound states. These are the same methods we use to deal with higher excited states. We have used the same notation here as is used in the code, for ease of verifying that the code follows the mathematical argument.

Let us extend our procedure to prove the first excited state is unique. We split up $(\sqrt{2}, \infty)$ into six pieces:
\begin{align*}
    I_1 = (1.40, 4.26),\ I_2 = (4.25, 4.43),\ I_3 = (4.42, 14.10),\\
    I_4 = (14.09, 14.12),\ I_5 = (14.11, 16.11),\ I_6 = (16.10, \infty).
\end{align*}
We apply the \FALL{} method to $I_1$, $I_3$, $I_5$, we apply the \BOUNDSTATEGOOD{} method to $I_2$, $I_4$, and we apply the \INFTYCROSSESMANY{} method to $I_6$. For the interval $I_4$, our careful stepping procedure as described above lets us find a time $T > 0$, for example $T = 8$, such that $y_b(T)$ crosses zero exactly once by time $T$ for all $b \in I_4$. We can also verify that $\dot y_b(T) > 0$, $\delta_b(T) > 0$, $\dot \delta_b(T) > 0$, so that the conditions of Lemma \ref{lem:above_does_cross} are satisfied and there is at most one first excited state in the interval $I_4$. Next we verify that the conditions of Lemma \ref{lem:cross_over_get_trapped} are satisfied uniformly for $b \in I_4$ at time $T$, so that there are no second or higher excited states in $I_4$. Altogether this shows that the first excited state exists and is unique, and sets us up to prove subsequent excited states are unique as well.

\subsection{Planning section}
We now describe the planning section of the code. Given a value $N \geq 0$, this section outputs a list of intervals $I_1, I_2, \ldots, I_k$, along with which method is to be used in each interval. The proving section will use this plan to verify that all bound states up to the $N^{th}$ (that is, all bound states with $\leq N$ zero crossings, or in other words the first $N$ excited states) are unique. 

Let $b_0, b_1, \ldots, b_N$ denote the locations of the first $N$ excited states (assuming for now that they are unique). We find their locations numerically with a binary search. To find $b_k$, we keep track of a lower bound $l < b_k$ and an upper bound $u > b_k$, and at each iteration, check how many times $y_m(t)$ crosses zero, $m = (l+u)/2$. If $y_m(t)$ crosses zero more than $k$ times, we set $u := m$, and if not we set $l := m$. We iterate until we have a small enough interval $(l, m)$ containing $b_k$.  

Next, we find small enough intervals around each bound state so that the \BOUNDSTATEGOOD{} method can run successfully for each bound state. We start with a large interval around $b_k$, width $~0.5$, and then keep on dividing the width by two until \BOUNDSTATEGOOD{} succeeds. 

Third, we fill in the space between the bound states with \FALL{} intervals. We include a buffer interval above the last bound state so as to make \INFTYCROSSESMANY{} run faster.

Finally, we create an interval $\pmb{\beta} = (0, \beta)$ corresponding to the infinite interval $(1/\beta, \infty)$ where we will show the ODE crosses zero at least $N+1$ times. 
\subsection{Proving section}
The proving section receives a list of intervals and methods from the planning section, and outputs a rigorous proof that the first $N$ excited states are unique. The first step is to verify that subsequent intervals intersect each other, so that every real number in the range $(\sqrt{2}, \infty)$ is covered by some interval. Next, the different methods are implemented as follows. 

The \FALL{} method receives an interval, e.g. $(1.4, 4.2)$, and must prove that for all $b$ in that interval $y_b(t)$ eventually has negative energy. It begins by attempting to integrate with that potentially very large input interval for $b$. Of course, \VNODELP{} will likely fail to integrate with such a large input interval. If this happens, we bisect the interval into two halves, an upper and lower half, and recursively apply the \FALL{} method to each half. Once the starting intervals are small enough, \VNODELP{} will successfully integrate and prove that the energy is eventually negative. This bisection method allows us to use larger intervals away from the bound states and smaller intervals closer to the bound states, where the computations are more delicate.

The \BOUNDSTATEGOOD{} method receives an interval $I$ which supposedly contains an $n$th bound state. It must prove that there is at most one $n$th bound state, no lower bound states, and no higher bound states in $I$. We use a careful stepping procedure to find some time $T > 0$ such that for all $b \in I$, $y_b(t)$ crosses zero exactly $n$ times by time $T$. This already shows that $I$ doesn't contain any lower bound states. Increasing $T$ if necessary, we also verify that $\dot y_b(T)$, $\delta_b(T)$, and $\dot \delta_b(T)$ all have the opposite sign as $y_b(T)$ uniformly in $I$. As discussed earlier, Lemma \ref{lem:above_does_cross} then implies that there is at most one $n$th bound state in $I$. Finally, we verify that the conditions in Lemma \ref{lem:cross_over_get_trapped} apply at time $T$, so there are no higher bound states in $I$. Throughout we use interval arithmetic, never floating point arithmetic.

The \INFTYCROSSESMANY{} method works similarly to the \FALL{} method. We bisect the interval $(0, \beta)$ into smaller pieces, and in each of these small pieces we prove that $w(t)$ has at least $N+1$ crossings.

Altogether, these methods show that if it exists, the $n$th bound state (counting from $n = 0$) must be unique and lie in the $n$th \BOUNDSTATEGOOD{} interval. This proves that all bound states up to the $N^{th}$ are unique, as desired. In fact, the code may also be used to show these bound states exist by counting crossing numbers, but this is already known by synthetic methods \cite{HMcL}. 

\section{Using \VNODELP{} and the data} 
\label{sec:numerics}
The code and full output logs from the proof procedure can be found at \url{https://github.com/alexander-cohen/NLKG-Uniqueness-Prover}, the most recent commit at the time of writing is 9cf63c06ca1838e64dd35fe11ca4fdfd45591714.
The code is contained in the single \texttt{C++} file ``nlkg\_uniqueness\_prover.cc'', and output logs are titled ``uniqueness\_output\_N=*.txt''. The code proved the first 20 excited states are unique in $\sim4$h running on a MacBook Pro 2017, 2.5 GHz. Time is the main limiting factor to proving uniqueness of more excited states.
We summarize the output of the proof for $N = 3$ excited states. Intervals are rounded for space, in actuality they have a nonempty intersection. See Figure~\ref{fig:pf_method_graph} for a visual representation of the same information.

\begin{center}
\noindent
\begin{longtable}[t]{r|c|l}
   Interval & Method & Details \\ 
    {[1.414, 4.266]} & \FALL{} & \\[0.5em]
    {[4.266, 4.433]} & \BOUNDSTATEGOOD{} & 
    \begin{tabular}[t]{l}
        Bound state 0, used $T=1.921$ \\ 
        $y(T) \in [0.127, 0.277]$ \\ 
        $\dot y(T) \in [-0.342, -0.283]$ \\ 
        $\delta \in [-0.374, -0.269]$ \\ 
        $\dot \delta \in [-0.139, -0.049]$
    \end{tabular}
    \\[0.5em]
    {[4.433, 14.095]} & \FALL{} & \\[0.5em]
    {[14.085, 14.115]} & \BOUNDSTATEGOOD{} & 
    \begin{tabular}[t]{l}
        Bound state 1, used $T=2.855$ \\ 
        $y(T) \in -0.3[40, 43]$ \\ 
        $\dot y(T) \in 0.452[46,55]$ \\ 
        $\delta \in 0.15[59,63]$ \\ 
        $\dot \delta \in 0.00[05,12]$
    \end{tabular}\\[0.5em]
    {[14.115, 29.090]} & \FALL{} & \\[0.5em]
    {[29.090, 29.174]} & \BOUNDSTATEGOOD{} & 
    \begin{tabular}[t]{l}
        Bound state 2, used $T=4.970$ \\ 
        $y(T) \in 0.1[05,29]$ \\ 
        $\dot y(T) \in -0.1[34,46]$ \\ 
        $\delta \in -0.1[18,26]$ \\ 
        $\dot \delta \in -0.0[57, 65]$
    \end{tabular}\\[0.5em]
    {[29.174, 49.339]} & \FALL{} & \\[0.5em]
    {[49.339, 49.381]} & \BOUNDSTATEGOOD{} & 
    \begin{tabular}[t]{l}
        Bound state 3, used $T=5.908$ \\ 
        $y(T) \in -0.1[68, 76]$ \\ 
        $\dot y(T) \in [0.198, 0.202]$ \\ 
        $\delta \in 0.07[31,55]$ \\ 
        $\dot \delta \in 0.01[72,91]$
    \end{tabular}\\[0.5em]
    {[49.381, 51.381]} & \FALL{} & \\[0.5em]
    {[0.000, 0.019]} & \INFTYCROSSESMANY{} & 
\end{longtable}
\end{center}

\begin{figure}[H]
    \centering
    \includegraphics[width=0.9\textwidth, trim={0 0  0 1.2cm}, clip]{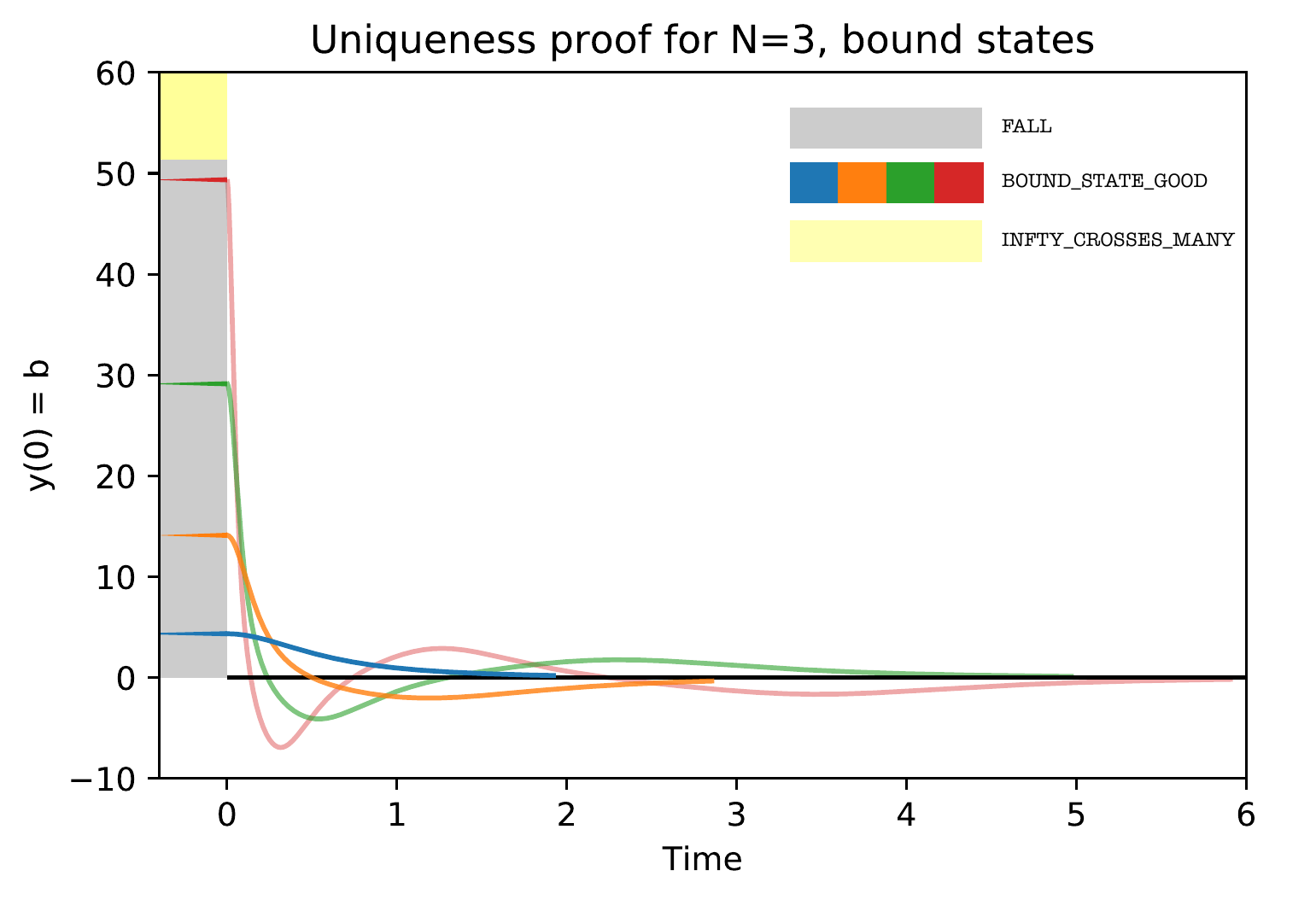}
    \caption{Graph showing first three excited states, and how the $b$-axis is partitioned by different proof methods.}
    \label{fig:pf_method_graph}
\end{figure}

\printbibliography

\typeout{get arXiv to do 4 passes: Label(s) may have changed. Rerun}
\end{document}